\documentclass[11pt]{amsart}

\usepackage{amsmath,amsthm,amssymb}
\usepackage{geometry}
\usepackage{mathrsfs}
\usepackage{color}

\newtheorem*{convdual}{Convex duality}

\newtheorem{teo}{Theorem}
\newtheorem*{teo*}{Main Theorem}
\newtheorem{lm}[teo]{Lemma}
\newtheorem{prop}[teo]{Proposition}
\newtheorem{coro}[teo]{Corollary}
\theoremstyle{definition}
\newtheorem{exa}[teo]{Example}
\newtheorem{defi}[teo]{Definition}
\newtheorem{oss}[teo]{Remark}

\newtheorem*{ack}{Acknowledgements}

\newcommand{\op}[1]{{\rm{#1}}}

\numberwithin{equation}{section}
\numberwithin{teo}{section}

\author{Lorenzo Brasco}
\address{{\bf L. B.}, Laboratoire d'Analyse, Topologie, Probabilit\'es, Aix-Marseille Universit\'e, 39 Rue Fr\'ed\'eric Joliot Curie, 13453 Marseille Cedex 13, France}
\email{brasco@cmi.univ-mrs.fr}

\author{Mircea Petrache}
\address{{\bf M. P.}, ETH, Departement Mathematik, R\"amistrasse 101, 
8092 Z\"urich, Switzerland}
\email{mircea.petrache@math.ethz.ch}

\keywords{Monge-Kantorovich problem, Beckmann problem, Smirnov Theorem, flat norm}
\subjclass[2010]{49K20, 46E35}

\title{A continuous model of transportation revisited}

\begin{document}

\maketitle
\begin{abstract}
We review two models of optimal transport, where congestion effects during the transport can be possibly taken into account. The first model is Beckmann's one, where the transport activities are modeled by vector fields with given divergence. The second one is the model by Carlier et al. (SIAM J Control Optim 47: 1330--1350, 2008), which in turn is the continuous reformulation of Wardrop's model on graphs.
We discuss the extensions of these models to their natural functional analytic setting and show that they are indeed equivalent, by using Smirnov decomposition theorem for normal $1-$currents.
\end{abstract}

\tableofcontents

\section{Introduction}
\subsection{Theoretical background} The present work is motivated by the study of transport problems for distributions started in \cite{BBD} and \cite{CJS}, which we desire to try and connect to related works in the theory of currents presented in \cite{AK,PS,S}. One of the important motivations for this problem is the study of distributions of the form
$$
\sum_{i=1}^\infty(\delta_{P_i}-\delta_{Q_i}),\qquad \mbox{ with } \qquad\sum_{i=1}^\infty|P_i-Q_i|<\infty,
$$
as in \cite{Po}. Such distributions arise as topological singularities in several geometric variational problems as described for example in \cite{Brez,HR,PR,Riv,sandier}.
\par
To start with, we formally define three variational problems which can be settled (for simplicity) on the closure of an open convex subset $\Omega\subset\mathbb{R}^N$ having smooth boundary.
For the moment we are a little bit imprecise about the datum $f$ but we will properly settle our hypotheses later.
The first problem is the minimization of the total variation of a Radon vector measure under a divergence constraint: 
\begin{displaymath}
\label{1beckmann}
\tag{$\mathcal{B}$}
\min_{V}\left\{\int_\Omega d|V|\,:\,\:-\op{div}\,V=f,\ V\cdot \nu_\Omega=0\right\}.
\end{displaymath}
The above problem can be connected by duality with the following one, called the Kantorovich problem:
\begin{displaymath}
 \label{eq:kantorovich}
\tag{$\mathcal{K}$}
\max_{\phi}\left\{\langle f,\phi\rangle: \|\nabla\phi\|_{L^\infty(\Omega)}\leq 1\right\},
\end{displaymath}
where now the variable $\phi$ is a Lipschitz function and $\langle\cdot,\cdot\rangle$ represents a suitable duality pairing. Finally the third problem is the minimization of the total length
\begin{displaymath}
\label{eq:transport}
\tag{$\mathcal{M}$}
\min_Q\left\{\int_{\mathcal P}\ell(\gamma)\,d Q(\gamma)\, :\, (e_1 - e_0)_\#Q=f\right\},
\end{displaymath}
where $\mathcal P$ is the space of Lipschitz continuous paths $\gamma:[0,1]\to \Omega$, the {\it length functional} $\ell$ is defined by
\[
\ell(\gamma)=\int_0^1 |\gamma'(t)|\, dt,
\] 
$e_0,e_1$ are the evaluation functions giving the starting and ending points of a path and the variable $Q$ is a measure concentrated on $\mathcal P$.
\par
The classical setting for the above problems is when $f$ is of the form $f=f^+-f^-$ where $f^+$ and $f^-$ are positive measures on $\Omega$ having the same mass (for example conventionally one can consider them to be probability measures). We point out that in this case a more familiar formulation of \eqref{eq:transport} is the so-called {\it Monge-Kantorovich problem}
\begin{displaymath}
\label{MK}
\tag{$\mathcal{M}'$}
\min_\eta\left\{\int_{\Omega\times\Omega} |x-y|\, d\eta(x,y)\, :\, (\pi_x)_\# \eta=f^+\quad\mbox{ and }\quad (\pi_y)_\# \eta=f^-\right\},
\end{displaymath}
where $\pi_x,\pi_y:\Omega\times\Omega\to\Omega$ stand for the projections on the first and second variable, respectively. It is useful to recall that the link between \eqref{eq:transport} and \eqref{MK} is given by the fact that if $\eta_0$ is optimal for Monge-Kantorovich problem then the measure which concentrates on {\it transport rays} i.e.
\[
Q_0=\int \delta_{\overline{x\,y}}\, d\eta_0(x,y),\qquad \mbox{ where }\ \overline{x\,y} \ \mbox{ stands for the segment connecting $x$ and $y$},
\]
is optimal in \eqref{eq:transport} and
\[
\int_{\mathcal{P}} \ell(\gamma)\, dQ_0(\gamma)=\int_{\Omega\times\Omega} |x-y|\, d\eta_0(x,y).
\]
When $f$ has the above mentioned form $f^+ - f^-$ the equivalence of the three problems above is well understood. The equivalence of \eqref{MK}=\eqref{eq:transport} and \eqref{eq:kantorovich} is the classical Kantorovich duality (see \cite{Ka}), while that between \eqref{1beckmann} and \eqref{eq:kantorovich}  seems to have been first identified in \cite{St}.
\par
Recently the equivalence of the above three problems has been shown in \cite{BBD} for $f$ belonging to a wider class, i.e. when $f$ is in the completion of the space of zero-average measures with respect to the norm dual to the $C^1$ (or flat) norm. This wider space was studied in \cite{Hanin} and characterized recently in \cite{BBD,BCJ}. A different point of view is also available in \cite{Maly}, where the space of such $f$ is called $W^{-1,1}$.

\subsection{Goals of the paper}
Our starting observation is that problem \eqref{1beckmann} pertains to a wide class of optimal transport problems introduced by Martin J. Beckmann in \cite{Be}, which are of the form
\begin{displaymath}
\label{Hbeckmann}
\tag{$\mathcal{B}_\mathcal{H}$}
\min_V\left\{\int_\Omega \mathcal{H}(V)\, dx\, :\, -\mathrm{div}\, V=f,\, V\cdot \nu_\Omega=0\right\},\qquad \mbox{ where }\ c_1\,|z|^p\le\mathcal{H}(z)\le c_2\,|z|^p,
\end{displaymath}
for a suitable density-cost convex function $\mathcal{H}:\mathbb{R}^N\to\mathbb{R}^+$ and $p\ge 1$. For a problem of this type the question of finding equivalent formulations of the form \eqref{eq:kantorovich} and \eqref{eq:transport} has already been addressed in \cite{BCS} (see also \cite{BraC2}) under some restrictive assumptions on $f$, like for example 
$$
f=f^+-f^-\qquad \mbox{ with}\quad  f^+,f^-\in L^p(\Omega)\quad \mbox{ and }\quad \int_\Omega f^+=\int_\Omega f^-=0.
$$
The goal of this paper is to complement and refine this analysis, first of all by studying problem \eqref{Hbeckmann} in its natural functional analytic setting, i.e. when $f$ belongs to a dual Sobolev space $W^{-1,p}$ (whose elements are not measures, in general). By expanding the analysis in \cite{BraC2,BCS} we will also see that alternative formulations of the type \eqref{eq:kantorovich} and \eqref{eq:transport} are still possible for \eqref{Hbeckmann} in this extended setting. These formulations are still well-posed on the dual space $W^{-1,p}$ and equivalence can be proved in this larger space. The problem corresponding to \eqref{eq:kantorovich} will now have the form (see Section 3 for more details)
\begin{displaymath}
\label{Hkantorovich}
\tag{$\mathcal{K}_\mathcal{H}$}
\max_{\phi}\, \left[\langle f,\phi\rangle -\int_\Omega \mathcal{H}^*(\nabla \phi)\, dx\right],
\end{displaymath}
and the equivalence with \eqref{Hbeckmann} will just follow by standard convex duality arguments (which are later recalled, for the convenience of the reader). On the contrary, in the proof of the equivalence between \eqref{Hbeckmann} and its Lagrangian formulation
\begin{displaymath}
\label{Htransport}
\tag{$\mathcal{M}_\mathcal{H}$}
\min_{Q}\, \left\{\int_\Omega \mathcal{H}(i_Q)\, dx\, :\, (e_1-e_0)_\# Q=f\right\},
\end{displaymath}
some care is needed and we will require to $f$ be a finite measure belonging to $W^{-1,p}$. Here the measure $i_Q$ will be some sort of {\it transport density}\footnote{When $\mathcal{H}(t)=|t|$ problem \eqref{Htransport} is again the Monge-Kantorovich one and $i_Q$ for an optimal $Q$ is nothing but the usual concept of transport density, see \cite{BB,DP1,FM}.} generated by $Q$, which takes into account the amount of work generated in each region by our distribution of curves $Q$ (see Section \ref{sec:lagrangian} for the precise definition). 
In particular the proof of this equivalence will point out another not emphasized connection to Geometric Measure Theory.
\par
The main result of this paper can be formulated as follows (see Theorems \ref{lm:duality} and \ref{th:wardbeck} for more precise statements):
\begin{teo*}
Let $1<p<\infty$. Suppose $\Omega\subset\mathbb R^N$ is the closure of a smooth bounded open set, let $f\in W^{-1,p}(\Omega)$ and let $\mathcal H$ be a strictly convex function having $p-$growth. Then the minimum in ($\mathcal B_{\mathcal H}$) and the maximum in ($\mathcal K_{\mathcal H}$) are achieved and coincide. Moreover the unique minimizer $V$ of ($\mathcal B_{\mathcal H}$) and any maximizer $v$ of ($\mathcal K_{\mathcal H}$) are linked by the relation $\nabla v\in\partial\mathcal{H}(V)$, as specified in Theorem \ref{lm:duality}.
\par
If in addition $f$ is a Radon measure then we have the following relationship among the optimizers of ($\mathcal B_{\mathcal H}$) and ($\mathcal M_{\mathcal H}$):
\begin{enumerate}
\item[(i)] the unique minimizer of ($\mathcal B_{\mathcal H}$) corresponds to a minimizer of ($\mathcal M_{\mathcal H}$) in the sense of Proposition \ref{smirnovlp};
\vskip.2cm
 \item[(ii)] each minimizer of ($\mathcal M_{\mathcal H}$) corresponds to the unique minimizer of ($\mathcal B_{\mathcal H}$) in the sense of Proposition \ref{smirnovlp}.
\end{enumerate}
\end{teo*}
The connection of the above theorem to Geometric Measure Theory lies in the basic theory of \emph{normal $1-$currents}, whose basic steps are recalled in the (long) appendix at the end of the paper. Indeed, in order to show equivalence of  \eqref{Hbeckmann} and \eqref{Htransport} 
our cornerstone is {\it Smirnov decomposition theorem} for $1-$currents. 
\vskip.2cm\par
For the sake of completeness and in order to neatly motivate the studies performed in this paper it is worth recalling  that the proof of this equivalence in \cite{BCS} was based on the {\it Dacorogna-Moser construction} to produce transport maps (see \cite{DM}), which has revealed to be a powerful tool for optimal transport problems\footnote{It is worth remarking that the first proof of the existence of an optimal transport map for problem \eqref{MK}, more than 200 years after Monge stated it, was based on a clever use of this construction (see \cite{EG}).}. In a nutshell, this method consists in associating to the ``static'' vector field $V$ which is optimal for \eqref{Hbeckmann} the following dynamical system
\[
\partial \mu_t+\mathrm{div}\left(\frac{V}{(1-t)\, f^++t\, f^-}\, \mu_t\right)=0,\qquad \mu_0=f^+,
\]
i.e. a continuity equation with driving velocity field $\widetilde V_t$ given by $V$ rescaled by the linear interpolation between $f^+$ and $f^-$. Assuming that one can give a sense (either deterministic or probabilistic) to the flow of $\widetilde V_t$, the construction of the measure $Q_V$ concentrated on the flow lines of $\widetilde V_t$ paves the way to the equivalence between the Lagrangian model \eqref{Htransport} and \eqref{Hbeckmann} (see \cite{BraC2,BCS} for more details). 
 
\subsection{Plan of the paper} 
In Section \ref{sec:well} we describe the function space $W^{-1,p}(\Omega)$ and we prove the existence of a minimizer for $(\mathcal B_{\mathcal H})$. Section \ref{sec:dual} treats the equivalence of $(\mathcal B_{\mathcal H})$ with $(\mathcal K_{\mathcal H})$ by appealing to classical convex analysis results. The aim of Section \ref{sec:lagrangian} is to introduce the Lagrangian counterpart of Beckmann's model and to show that the two models are equivalent. 
A self-contained Appendix complements the paper. There we introduce relevant concepts from Geometric Measure Theory and we translate Smirnov's decomposition theorem into the language of $L^1$ vector fields.

\section{Well-posedness of Beckmann's problem}

\label{sec:well}

Let $\Omega\subset\mathbb{R}^N$ be the closure of an open bounded connected set having smooth boundary. In what follows $\Omega$ {\it will always be compact}.
Given $1<q<\infty$ we indicate with $W^{1,q}(\Omega)$ the usual Sobolev space of $L^q(\Omega)$ functions whose distributional gradient is in $L^q(\Omega;\mathbb{R}^N)$ as well. We then define the quotient space
\[
\dot W^{1,q}(\Omega)=\frac{W^{1,q}(\Omega)}{\sim},
\]
where $\sim$ is the equivalence relation defined by
\[
u\sim v\quad \Longleftrightarrow\quad \mbox{ there exists } c\in\mathbb{R} \mbox{ such that }\quad u(x)-v(x)=c\quad \mbox{ for a.e. }x\in\Omega.
\]
When needed the elements of $\dot W^{1,q}(\Omega)$ will be identified with functions in $W^{1,q}(\Omega)$ having zero mean.
We endow the space $\dot W^{1,q}(\Omega)$ with the norm
\[
\|u\|_{\dot W^{1,q}(\Omega)}:=\left(\int_\Omega |\nabla u|^q\, dx\right)^\frac{1}{q},\qquad \dot u\in W^{1,q}(\Omega),
\]
then we denote by $\dot W^{-1,p}(\Omega)$ its dual space, equipped with the dual norm. The latter is defined as usual by
\[
\|T\|_{\dot W^{-1,p}(\Omega)}:=\sup \left\{\langle T,\varphi\rangle\, :\,\varphi\in \dot W^{1,q}(\Omega),\, \|\varphi\|_{\dot W^{1,q}}=1\right\},
\]
where $p=q/(q-1)$.
We start recalling the following basic fact.
\begin{lm}
\label{lm:norma}
Let $T\in \dot W^{-1,p}(\Omega)$. Then
\[
\|T\|_{\dot W^{-1,p}(\Omega)}=p^\frac{1}{p}\,\left[\max_{\varphi\in \dot W^{1,q}(\Omega)} |\langle T,\varphi\rangle|-\frac{1}{q}\, \int_\Omega |\nabla \varphi|^q\, dx\right]^\frac{1}{p}.
\]
\end{lm}
\begin{proof}
For every $\varphi\in \dot W^{1,q}(\Omega)$ we have
\[
|\langle T,\varphi\rangle|-\frac{1}{q}\, \int_\Omega |\nabla \varphi|^q\, dx\le \sup_{\lambda\ge 0} \left[\lambda\, |\langle T,\varphi\rangle|-\frac{\lambda^q}{q}\, \int_\Omega |\nabla \varphi|^q\, dx\right].
\]
On the other hand the supremum on the right is readily computed: this corresponds to the choice
\[
\lambda=|\langle T,\varphi\rangle|^\frac{1}{q-1}\, \left(\int_\Omega |\nabla \varphi|^q\, dx\right)^{-\frac{1}{q-1}},
\]
which gives
\[
\sup_{\lambda\ge 0} \left[\lambda\, |\langle T,\varphi\rangle|-\frac{\lambda^q}{q}\, \int_\Omega |\nabla u|^q\, dx\right]=\frac{1}{p}\, \left(\frac{|\langle T,\varphi\rangle|}{\|\varphi\|_{\dot W^{1,q}}}\right)^p.
\]
Passing to the supremum over $\varphi\in \dot W^{1,q}(\Omega)$ and using the definition of the dual norm we get the thesis.
\end{proof}
We also denote by $\mathcal{E}'_1(\Omega)$ the space of distributions of order $1$ with (compact) support in $\Omega$. In what follows we tacitly identify this space with the dual of the space $C^1(\Omega)$, endowed with the norm
\[
\|\varphi\|_{C^1(\Omega)}=\sup_{x\in\Omega}|\varphi(x)|+\sup_{x\in\Omega}|\nabla \varphi(x)|.
\]
We denote by $\nu_\Omega$ the outer normal unit vector to $\partial\Omega$. We have the following characterization for the dual space $\dot W^{-1,p}(\Omega)$.  
\begin{lm}
\label{lm:cara}
Let $p=q/(q-1)$. We say that a vector field $V\in L^p(\Omega;\mathbb{R}^N)$ and $T\in \mathcal{E}'_1(\Omega)$ satisfies 
\begin{equation}
\label{neumann}
-\mathrm{div\,}V=T\quad \mbox{ in } \Omega, \qquad V\cdot \nu_\Omega=0\quad \mbox{ on } \partial\Omega,
\end{equation}
if
\[
\int_\Omega \nabla \varphi\cdot V\, dx=\langle T,\varphi\rangle,\qquad \mbox{ for every } \varphi \in C^1({\Omega}).
\]
If we set
\[
\mathcal{E}'_{1,p}(\Omega)=\{T\in \mathcal{E}'_1(\Omega)\, :\, \mbox{ there exists } V\in L^p(\Omega;\mathbb{R}^N) \mbox{ satisfying } \eqref{neumann}\},
\]
we then have the identification
\[
\dot W^{-1,p}(\Omega)=\mathcal{E}'_{1,p}(\Omega).
\]
\end{lm}
\begin{proof}
Let $T\in \dot W^{-1,p}(\Omega)$. We observe that then $T\in\mathcal{E}'_1(\Omega)$ as well. Now consider the following maximization problem
\[
\sup_{v\in \dot W^{1,q}(\Omega)} \langle T,v\rangle -\frac{1}{q}\, \int_\Omega |\nabla v|^q\, dx.
\]
By means of the Direct Methods, it is not difficult to see that there exists a (unique) maximizer $u\in \dot W^{1,q}(\Omega)$ for this problem. Moreover such a maximizer satisfies the relevant Euler-Lagrange equation given by
\[
\int_\Omega |\nabla u|^{q-2}\, \nabla u\cdot \nabla \varphi \, dx=\langle T,\varphi\rangle,\qquad \mbox{ for every } \varphi\in \dot W^{1,q}(\Omega).
\] 
By taking $V=|\nabla u|^{q-2}\, \nabla u\in L^p(\Omega;\mathbb{R}^N)$, the previous identity implies $T\in \mathcal{E}'_{1,p}(\Omega)$.
\vskip.2cm\noindent
Conversely let us take $T\in \mathcal{E}'_{1,p}(\Omega)$. Then for every $\varphi \in C^1({\Omega})$ equation \eqref{neumann} implies
\[
|\langle T,\varphi\rangle|=\left|\int_\Omega \nabla \varphi\cdot V\, dx\right|\le \|\varphi\|_{\dot W^{1,q}}\, \|V\|_{L^p(\Omega)}.
\]
Using the density of $C^1(\Omega)$ in $\dot W^{1,q}(\Omega)$ we obtain that $T$ can be extended in a unique way as an element (that we still denote $T$ for simplicity) of $\dot W^{-1,q}(\Omega)$. This extension satisfies
\[
\|T\|_{\dot W^{-1,p}(\Omega)}\le \|V\|_{L^p(\Omega)},
\]
as can be seen by taking the supremum in the previous inequality.
\end{proof}
\begin{oss}
We remark that the elements of $\mathcal{E}'_{1,p}(\Omega)$ have ``zero average'' i.e.
\[
\langle T,1\rangle=0,
\]
as follows by testing the weak formulation of \eqref{neumann} with $\varphi\equiv 1$. This is coherent with the previous identification $\dot W^{-1,p}(\Omega)=\mathcal{E}'_{1,p}(\Omega)$ since by construction the space $\dot W^{1,q}(\Omega)$ does not contain any non trivial constant function.
\end{oss}
\begin{exa}
 \label{curvette1}
Consider the measure $T=\delta_a-\delta_b$ for two points $a\neq b\in\mathbb R^N$. We claim that 
$$
T=\delta_a-\delta_b\in \dot W^{-1,p}(\Omega)\qquad \mbox{ if and only if }\qquad 1\le p<N/(N-1),
$$ 
where $\Omega$ is a sufficiently large ball containing $a,b$ in its interior.
We prove this by using the characterization of Lemma \ref{lm:cara}. 
\vskip.2cm\noindent
Suppose indeed that there exists some $V\in L^p(\Omega)$ such that $-\op{div}\,V=T$. We pick a ball $B_r(a)$ centered at $a$ and having radius $r$ such that $2\,r<|a-b|$. Then for each $\varepsilon< r$ we consider a $C^1_0(B_r(a))$ function $\eta_\varepsilon$ such that
\[
\eta_\varepsilon\equiv 1\quad \mbox{ in } B_{r-\varepsilon}(a)\qquad \mbox{ and }\qquad \|\nabla \eta_\varepsilon\|_{L^\infty}\le C\, \varepsilon^{-1}.
\]
Thanks to our assumption we have
$$
1=\langle T,\eta_\varepsilon\rangle=\int_{B_r(a)} V\cdot \nabla \eta_\varepsilon\, dx,$$
so that 
\[
\int_{B_{r}(a)\setminus B_{r-\varepsilon}(a)} |V|\, dx\ge \frac{\varepsilon}{C}.
\]
By H\"older inequality this easily implies a lower bound on the $L^p$ norm of $V$, namely
\[
\begin{split}
\int_{B_r(a)} |V|^p\, dx\ge \varepsilon^p\, |{B_{r}(a)\setminus B_{r-\varepsilon}(a)}|^{1-p}&=C_{N,p}\, \varepsilon^{p}\, r^{N\,(1-p)}\, \left[1-\left(1-\frac{\varepsilon}{r}\right)^N\right]^{1-p}.\\
\end{split}
\]
We now make the choice $\varepsilon=r/2$, so that from the previous we can infer
\[
\int_{B_r(a)} |V|^p\, dx\ge \widetilde C_{N,p}\, r^{p+N\, (1-p)}=\widetilde C_{N,p}\, r^{N-p\, (N-1)}.
\]
The previous estimate clearly contradicts the assumption $V\in L^p(\Omega)$ if the exponent $N-p\,(N-1)$ is not strictly positive. Therefore we see by Lemma \ref{lm:cara} that $p<N/(N-1)$ is a necessary condition for $T\in \dot W^{-1,p}(\Omega)$.
\vskip.2cm\noindent
This condition on $p$ is also sufficient for $T$ to belong to $\dot W^{-1,p}(\Omega)$, as we now proceed to show. Set $2\,\tau=|a-b|$ and for simplicity assume that $a=(-\tau,0,\dots,0)$ and $b=(\tau,0,\dots,0)$. We use the notation $x=(x_1,x')$ for a generic point in $\mathbb{R}^N$, where $x'\in\mathbb{R}^{N-1}$. We define the following vector field 
\[
V_{a,b}(x)=\left\{\begin{array}{cc}
\displaystyle\frac{(x_1+\tau,x')}{(x_1+\tau)^N},&\mbox{ if } |x'|\le \tau \mbox{ and } |x'|-\tau\le x_1\le 0,\\
&\\
 \displaystyle\frac{(x_1-\tau,x')}{(x_1-\tau)^N},&\mbox{ if } |x'|\le \tau \mbox{ and } \tau-|x'|\ge x_1\ge 0,\\
&\\
(0,\dots,0),&\mbox{ otherwise}.
\end{array}
\right.
\]
It is easily seen that $\mathrm{div}\, V_{a,b}=\delta_a-\delta_b$ and that $V_{a,b}$ is supported on the set
\[
D_{a,b}=\left\{(x_1,x')\in\mathbb{R}^N\, :\, \frac{|a-b|}{2}\ge |x'|+|x_1|\right\},
\] 
which is just the the union of two cones centered at $a$ and $b$ having opening $1$ and height $\tau=|a-b|/2$.
By construction we have
\[
\begin{split}
\int_{Q_{a,b}} |V_{a,b}|^p\, dx&=2\, \int_{-\tau}^0 \int_{\{x'\, :\, |x'|=x_1+\tau\} }\displaystyle \frac{\left(\sqrt{(x_1+\tau)^2+|x'|^2}\right)^p}{(x_1+\tau)^{Np}}\, dx'\, dx_1\\
&=2^\frac{p+2}{2}\,N\,\omega_N \int_{-\tau}^0 (x_1+\tau)^{-Np+p+N-1}\, dx_1\\
\end{split}
\]
so that finally
$$
\|V_{a,b}\|_{L^p}^p\leq C_{N,p}\,|a-b|^{N-p(N-1)},
$$
thanks to our assumption $p<N/(N-1)$. For some related constructions the reader is referred to
\cite[Proposition 3.2]{BCPS} and \cite[Lemma 8.3]{PR}.
\end{exa}
As a consequence of Lemma \ref{lm:cara} we have the following well-posedness result for Beckmann's problem.
\begin{prop}
\label{existbeck}
Let $1<p<\infty$. Let $\mathcal{H}:\Omega\times\mathbb{R}^N$ be a Carath\'eodory function such that $z\mapsto \mathcal{H}(x,z)$ is convex on $\mathbb{R}^N$ for every $x\in\Omega$. We further suppose that $\mathcal{H}$ satisfies the growth conditions
\begin{equation}
\label{pgrowth_vect}
\lambda(|z|^p-1)\le \mathcal{H}(x,z)\le \frac{1}{\lambda} (|z|^p+1),\qquad (x,z)\in\Omega\times\mathbb{R}^N
\end{equation}
for some $0< \lambda\le 1$. Then the following problem
\begin{equation}
\label{beckmann}
\min_{V\in L^p(\Omega;\mathbb{R}^N)}\left\{\int_\Omega \mathcal{H}(x,V)\, dx\, :\, -\mathrm{div\,}V=T,\quad V\cdot\nu_\Omega=0\right\}
\end{equation}
admits a minimizer with finite energy if and only if $T\in \dot W^{-1,p}(\Omega)$.
\end{prop}
\begin{proof}
Let $T\in \dot W^{-1,p}(\Omega)$. Thanks to Lemma \ref{lm:cara} there exists at least one admissible vector field $V_0$ with finite energy, thus the infimum \eqref{beckmann} is finite. If $\{V_n\}_{n\in\mathbb{N}}\subset L^p(\Omega;\mathbb{R}^N)$ is a minimizing sequence then the hypothesis \eqref{pgrowth_vect} on $\mathcal{H}$ guarantees that this sequence is weakly convergent to some $\widetilde V\in L^p(\Omega;\mathbb{R}^N)$. Thanks to the convexity of $\mathcal{H}$ the functional is weakly lower semicontinuous, i.e.
\[
\begin{split}
\int_\Omega \mathcal{H}(x,\widetilde V)\, dx&\le \liminf_{n\to\infty} \int_\Omega \mathcal{H}(x,V_n)\, dx\\
&=\min_{V\in L^p(\Omega;\mathbb{R}^N)}\left\{\int_\Omega \mathcal{H}(x,V)\, dx\, :\, \begin{array}{c}-\mathrm{div\,}V=T,\\ V\cdot\nu_\Omega=0\end{array}\right\}.
\end{split}
\]
Moreover the vector field $\widetilde V$ is still admissible since
\[
\int_\Omega \nabla \varphi\cdot \widetilde V\, dx=\lim_{n\to\infty}\int_\Omega \nabla \varphi\cdot V_n\, dx=\langle T,\varphi\rangle,\qquad \mbox{ for every }\varphi\in C^1({\Omega}),
\]
by weak convergence. Therefore $\widetilde V$ realizes the minimum.
\vskip.2cm\noindent
On the other hand suppose that $T\not\in \dot W^{-1,p}(\Omega)$.  Again thanks to Lemma \ref{lm:cara} we have that the set of admissible vector fields is empty so the problem is not well-posed.
\end{proof}
We need the following definition.
\begin{defi}
\label{Vacyclic}
We say that a vector field $V\in L^1(\Omega;\mathbb R^N)$ is \emph{acyclic} if whenever we can write $V=V_1+V_2$ with $|V|=|V_1|+|V_2|$ and $\mathrm{div}\, V_1=0$ with homogeneous Neumann boundary condition, namely 
\[
\int_\Omega V_1\cdot\nabla\varphi\, dx=0,\qquad \mbox{ for every }\varphi\in C^1(\Omega),
\]
there must result $V_1 \equiv 0$.
\end{defi}
The following is a mild regularity result for optimizers of \eqref{beckmann} in the {\it isotropic} case i.e. when $\mathcal{H}$ depends on the variable $z$ only through its modulus. This becomes crucial in order to equivalently reformulate \eqref{beckmann} as a Lagrangian problem, where the transport is described by measures on paths. 
\begin{prop}
\label{acyclicbeck}
Assume that $\mathcal{H}$ satisfies the hypotheses of Proposition \ref{existbeck}. In addition assume that
\[
z\mapsto \mathcal{H}(x,z)\ \mbox{ is a strictly convex increasing function of $|z|$ for every $x$}.
\]
Then there exists a unique minimizer $V$ for \eqref{beckmann} and $V$ is acyclic.
\end{prop}
\begin{proof}
The uniqueness of $V$ follows by strict convexity. We now prove that $V$ is acyclic. 
Suppose that we can write $V=V_1+V_2$ for some vector fields $V_1,V_2\in L^1(\Omega;\mathbb{R}^N)$ such that
\[
|V|=|V_1|+|V_2|\qquad \mbox{ and }\qquad \mathrm{div}\,V_1=0.
\] 
It follows that $\op{div}\,V=\op{div}\,V_2$ and $|V|\geq |V_2|$. Thus $V_2$ is a competitor for problem \eqref{beckmann} with energy not larger than that of $V$ thanks to the monotonicity of $\mathcal H$. Since $V$ is the unique minimizer, it must have energy equal to that of $V_2$. Thus $|V|=|V_2|$ and $|V_1|=0$ almost everywhere. This shows that $V$ is acyclic, concluding the proof.
\end{proof}

\section{Duality for Beckmann's problem}
\label{sec:dual}

We need the following general convex duality result (for the proof the reader is referred to \cite[Proposition 5, page 89]{Ek}). The statement has been slightly simplified in order to be directly adapted to our setting.
\begin{convdual}
Let $\mathcal{F}:Y\to \mathbb{R}$ be a convex lower semicontinuous functional on the reflexive Banach space $Y$. Let $X$ be another reflexive Banach space and $A:X\to Y$ a bounded linear operator, with adjoint operator $A^*:Y^*\to X^*$. 
Then we have
\begin{equation}
\label{convdual}
\sup_{x\in X}\, \langle x^*,x\rangle -\mathcal{F}(A\,x)=\inf_{y^*\in Y^*}\{\mathcal{F}^*(y^*)\, :\, A^* y^*=x^*\},\qquad x^*\in X^*,
\end{equation}
where $\mathcal{F}^*:Y^*\to\mathbb{R}\cup\{+\infty\}$ denotes the Legendre-Fenchel transform of $\mathcal{F}$. If the supremum in \eqref{convdual} is attained at some $x_0\in X$ then the infimum in \eqref{convdual} is attained as well by a $y_0^*\in Y^*$ such that
\[
y^*_0\in\partial\mathcal{F}(A\,x_0).
\]
\end{convdual}
Thanks to the above result we obtain that Beckmann's problem admits a dual formulation which is a classical elliptic problem in Calculus of Variations. 
\begin{teo}[Duality]
\label{lm:duality}
Let $1<p<\infty$ and $q=p/(p-1)$. Let $\mathcal{H}$ be a function satisfying the hypotheses of Proposition \ref{existbeck} and $T\in \dot W^{-1,p}(\Omega)$. Then
\begin{equation}
\label{duality}
\begin{split}
\min_{V\in L^p(\Omega;\mathbb{R}^N)}\left\{\int_\Omega \mathcal{H}(x,V)\, dx\, \right.&:\left.\, \begin{array}{c}-\mathrm{div\,}V=T,\\ V\cdot\nu_\Omega=0\end{array}\right\}\\
&=\max_{v\in \dot W^{1,q}(\Omega)} \left\{\langle T,v\rangle
-\int_\Omega \mathcal{H}^*(x,\nabla v)\, dx\right\},
\end{split}
\end{equation}
where $\mathcal{H}^*$ is the partial Legendre-Fenchel transform of $\mathcal{H}$, i.e.
\[
\mathcal{H}^*(x,\xi)=\sup_{z\in\mathbb{R}^N} \xi\cdot z-\mathcal{H}(x,z),\qquad x\in\Omega,\,\xi\in\mathbb{R}^N.
\]
Moreover if $V_0\in L^p(\Omega)$ and $v_0\in \dot W^{1,q}(\Omega)$ are two optimizers for the problems in \eqref{duality} then we have the following primal-dual optimality condition
\begin{equation}
\label{primaldual}
V_0\in\partial \mathcal{H}^*(x,\nabla v_0)\qquad \mbox{ in }\Omega,
\end{equation}
where $\partial\mathcal{H}^*$ denotes the subgradient with respect to the $\xi$ variable, i.e.
\[
\partial\mathcal{H}^*(x,\xi)=\{z\in\mathbb{R}^N\, :\, \mathcal{H}^*(x,\xi)+\mathcal{H}(x,z)=\xi\cdot z\},\qquad  x\in\Omega.
\]
\end{teo}
\begin{proof}
To prove \eqref{duality} it is sufficient to apply the previous result with the choices
\[
Y=L^q(\Omega;\mathbb{R}^N),\quad X=\dot W^{1,q}(\Omega),\quad \mathcal{F}(\phi)=\int_\Omega \mathcal{H}^*(x,\phi(x))\, dx\quad \mbox{ and }\quad A\, (\varphi)=\nabla \varphi. 
\]
The operator $A$ is bounded since 
\[
\|A(\varphi)\|_Y=\|\nabla \varphi\|_{L^q(\Omega)}=\|\varphi\|_X,\qquad \mbox{for every }\varphi\in X,
\]
and
\[
\mathcal{F}^*(\xi)=\int_\Omega \mathcal{H}^{**}(x,\xi(x))\, dx=\int_\Omega \mathcal{H}(x,\xi(x))\, dx,
\]
since $\xi\mapsto \mathcal{H}(x,\xi)$ is convex and lower semicontinuous, for every $x\in\Omega$.
We only need to compute the adjoint operator $A^*:L^p(\Omega;\mathbb{R}^N)\to \dot W^{-1,p}(\Omega)$. Let us define the map $\Psi: L^p(\Omega;\mathbb{R}^N)\to \mathcal{E}'_{1,p}(\Omega)$ by
\[
\Psi(V)\in \mathcal{E}'_{1}(\Omega)\qquad \mbox{ such that } \langle \Psi(V),\varphi\rangle=\int_\Omega \nabla \varphi\cdot V\, dx,\quad \mbox{ for every }\varphi \in C^1({\Omega}).
\]
Observe that $\Psi$ is a linear operator whose image is contained in $\mathcal{E}'_{1,p}(\Omega)=\dot W^{-1,p}(\Omega)$  by construction and by the definition of $\mathcal{E}'_{1,p}(\Omega)$. Moreover for $\varphi\in C^1({\Omega})$ and $V\in L^p(\Omega;\mathbb{R}^N)$ we have
\[
\langle A\,\varphi,V\rangle=\int_\Omega \nabla \varphi\cdot V\, dx=\langle \varphi, \Psi(V)\rangle.
\]
By density of $C^1({\Omega})$ in $W^{1,q}(\Omega)$ we obtain that $\Psi=A^*$, thus \eqref{duality} follows from \eqref{convdual}.
\vskip.2cm\noindent
The primal-dual optimality condition \eqref{primaldual} is a direct consequence of the second part of the convex duality result as well. It is sufficient to observe that the maximum in \eqref{duality} is attained at some $v_0\in \dot W^{1,p}(\Omega)$ by the Direct Methods. Thus, by the above convex duality theorem, a minimizer $V_0$ of Beckmann's problem has to satisfy
\[
V_0\in\partial \mathcal{F}(\nabla v_0),
\]  
which implies directly \eqref{primaldual}.
\end{proof}
A significant instance of the previous result corresponds to $\mathcal{H}(x,z)=|z|^p$. Thanks to Lemma \ref{lm:norma} we have the following result.
\begin{coro}\label{carattdivLp}
For every $T\in \dot W^{-1,p}(\Omega)$ we have
\[
\|T\|_{\dot W^{-1,p}(\Omega)}=\min_{V\in L^p(\Omega;\mathbb{R}^N)} \Big\{\|V\|_{L^p(\Omega)}\, :\, -\mathrm{div\,}V=T,\quad V\cdot\nu_\Omega=0\Big\}.
\]
\end{coro}
\begin{proof}
It is sufficient to use \eqref{duality} and Lemma \ref{lm:norma} and to observe that
\[
\max_{\varphi\in \dot W^{1,q}(\Omega)} |\langle T,\varphi\rangle|-\frac{1}{q}\, \int_\Omega |\nabla \varphi|^q\, dx=\max_{\varphi\in \dot W^{1,q}(\Omega)} \langle T,\varphi\rangle-\frac{1}{q}\, \int_\Omega |\nabla \varphi|^q\, dx.
\]
This establishes the thesis.
\end{proof}
\begin{coro}
Under the hypotheses of Theorem \ref{lm:duality} we have that the functional
\[
\begin{array}{cccc}
\mathfrak{F}_\mathcal{H}:& \dot W^{-1,p}(\Omega)&\to& \mathbb{R}^+\\
& T &\mapsto & \mbox{ minimal value \eqref{beckmann}}
\end{array}
\]
is convex and weakly lower semicontinuous.
\end{coro}
\begin{proof}
It is sufficient to observe that thanks to Theorem \ref{lm:duality} the value \eqref{beckmann} can be written as a supremum of the affine continuous functionals $L_\varphi$ defined by
\[
L_\varphi(T)=\langle T,\varphi\rangle-\int_\Omega \mathcal{H}^*(x,\nabla \varphi)\, dx,\qquad \varphi\in \dot W^{1,q}(\Omega).
\]
Then the thesis follows.
\end{proof}
Some comments are in order about the duality result of Theorem \ref{lm:duality}.
\begin{oss}[Economic interpretation]
By the so-called {\it Legendre reciprocity formula} in Convex Analysis the primal-dual optimality condition \eqref{primaldual} can be equivalently written as
\begin{equation}
\label{primaldual2}
\nabla v_0\in \partial \mathcal{H}(x,V_0),\qquad \mbox{ in } \Omega,
\end{equation}
so this result is the rigorous justification of the necessary optimality conditions derived in \cite[Lemma 2]{Be}. Such $v_0$ is called a {\it Beckmann potential} and its economic interpretation is that of an efficiency price, i.e. it represents a system price for moving commodities in the most efficient regime for a transport company. It can be seen as a generalization of a Kantorovich potential to a situation where the cost to move some unit of mass from $x$ to $y$ is not fixed. Indeed it depends on the quantity of traffic generated by the transport $V_0$ itself. Heuristically observe that in this case the minimal cost is given by the ``congested metric''
\[
d_{V_0}(x_0,x_1)=\min_{\gamma\, :\, \gamma(i)=x_i} \int_{0}^1 \left|\nabla \mathcal{H}(\cdot,V_0)\circ \gamma|\,|\gamma'(t)\right|\, dt.
\]
In other words each mass particle is charged for the marginal cost it produces, the latter being the derivative of the function $\mathcal{H}$ (we suppose for simplicity that $\mathcal{H}$ possesses a true gradient and not just a subgradient). Then $v_0$ acts as a Kantorovich potential for the Optimal Transport problem
\[
\min\left\{\int_{\Omega\times\Omega} d_{V_0}(x,y)\, d\eta(x,y)\, :\, (\pi_x)_\#\eta=T_+\quad\mbox{ and }\quad (\pi_y)_\#\eta=T_-\right\},
\]
where we assume for simplicity that $T=T_+-T_-$, with $T_+$ and $T_-$ positive measures having equal total masses.
It should be remarked that $\nabla \varphi_0$ does not give the direction of optimal transportation in Beckmann's problem since $\nabla \varphi_0$ and $V_0$ are only linked through the relation \eqref{primaldual2} and they are not parallel in general. They are guaranteed to be parallel only when the cost function $\mathcal{H}$ is {\it isotropic}, i.e. when it just depends on $|V|$ for every admissible vector field $V$. This is the case studied by Beckmann in his original paper \cite{Be}.
\end{oss}
\begin{oss}[Regularity of optimal vector fields]
We point out that if $z\mapsto \mathcal H(x,z)$ is strictly convex then $\xi\mapsto \mathcal H^*(x,\xi)$ is $C^1$. In this case the optimal $V_0$ is unique and we have
\[
V_0=\nabla \mathcal{H}^*(x,\nabla v_0).
\]
Then the regularity of the optimal vector field $V_0$ can be recovered from the regularity of a Beckmann potential, which solves the following elliptic boundary value problem 
\begin{equation}
\label{EL}
\left\{\begin{array}{cccc}
-\mathrm{div\,} \nabla \mathcal{H}^*(x,\nabla u)&=&T,& \mbox{ in }\Omega\\
\nabla \mathcal{H}^*(x,\nabla u)\cdot \nu_\Omega&=&0, & \mbox{ on }\partial\Omega.
\end{array}
\right.
\end{equation}
For instance if $\mathcal{H}^*$ in uniformly convex ``at infinity'', meaning that there exist $ C_1,C_2, M>0$ such that
\[
C_1\, (1+|z|^2)^\frac{q-2}{2} \le \min_{|\xi|=1}\, \langle D^2 \mathcal{H}^*(x,z)\, \xi,\xi\rangle,\qquad \mbox{ for every } |z|\ge M,\, x\in\Omega,
\]
and such that
\[
|D^2 \mathcal{H}^*(x,z)|\le C_2\, (1+|z|^2)^\frac{q-2}{2},\qquad (x,z)\in\Omega\times\mathbb{R}^N,
\]
then $V$ is bounded provided that $T\in L^{N+\varepsilon}(\Omega)$, with $\varepsilon>0$. Indeed, in this case solutions to \eqref{EL} are Lipschitz. These assumptions are verified for example by (see \cite{BCS})
\[
\mathcal{H}^*(z)=\frac{1}{q} (|z|-\delta)^q_+,\qquad z\in\mathbb{R}^N, 
\]
where $(\cdot)_+$ stands for the positive part and where we assume $\delta\ge 0$, but they are violated by anisotropic functions of the type
\[
\mathcal{H}^*(z)=\sum_{i=1}^N \frac{1}{q}\, (|z_i|-\delta_i)^q_+,\qquad z=(z_1,\dots,z_N)\in\mathbb{R}^N,
\]
considered for example in \cite{BraC2,BraC}.
\end{oss}

\section{A Lagrangian reformulation}
\label{sec:lagrangian}

The aim of this section is to introduce a Lagrangian counterpart of Beckmann's model and to show how the two models turn out to be equivalent. The model we are going to present is a continuous version of a classical discrete model on networks by Wardrop (see \cite{Wa}). This continuous model has already been addressed in \cite{CJS} and the equivalence has been discussed in \cite{BCS}.  
We prove well-posedness of the Lagrangian problem and equivalence of the models by imposing in addition that the datum $T$ is a finite measure belonging to $\dot W^{-1,p}(\Omega)$. The proofs use Smirnov's decomposition theorem for $1-$currents (see Theorem \ref{smirnov}). 
\vskip.2cm\noindent
Given two Lipschitz curves $\gamma_1,\gamma_2:[0,1]\to\Omega$ we say that they are {\it equivalent} if there exists a continuous surjective nondecreasing function $\mathfrak{t}:[0,1]\to[0,1]$ such that
\[
\gamma_2(t)=\gamma_1(\mathfrak{t}(t)),\qquad \mbox{ for every } t\in[0,1].
\]
We call $\mathcal{L}(\Omega)$ the set of all equivalence classes of Lipschitz paths in $\Omega$. We introduce a topology on this set by defining the following distance
$$
d(\gamma_1,\gamma_2):=\max\left\{|\hat\gamma_1(t) - \hat\gamma_2(t)|:\:t\in[0,1],\:\hat\gamma_i\text{ equivalent to }\gamma_i\right\}.
$$
Observe that convergence in this metric is nothing but the usual uniform convergence, up to reparameterizations.
\vskip.2cm\noindent
We denote the class of finite positive Borel (with respect to the above topology) measures on $\mathcal{L}(\Omega)$ by $\mathscr{M}_+(\mathcal L(\Omega))$. For $Q\in\mathscr{M}_+(\mathcal L(\Omega))$. We define the corresponding {\it traffic intensity} by
\[
\langle i_Q,\varphi\rangle:=\int_{\mathcal{L}(\Omega)} \left(\int_0^1 \varphi(\gamma(t))\,|\gamma'(t)|\, dt\right)\, dQ(\gamma),\qquad \varphi\in C({\Omega}),
\]
provided that the outer integral converges, in which case we say that ``the traffic intensity $i_Q$ exists''. If this is the case then the following integral also converges:
\[
\langle \mathbf{i}_{Q},\varphi\rangle=\int_{\mathcal{L}(\Omega)} \left(\int_0^1 \varphi(\gamma(t))\cdot \gamma'(t)\, dt\right)\, dQ(\gamma),\qquad \varphi\in C({\Omega};\mathbb{R}^N).
\]
These definitions do not depend on the particular representative of the equivalence class we chosen, since the integrals in brackets are invariant under time reparameterization.
\begin{oss}
Observe that $i_Q$ counts in a scalar way the traffic generated by $Q$ while $\mathbf{i}_Q$ computes it in a vectorial way. This means that in principle $i_Q$ and $|\mathbf{i}_Q|$ could be very different: in $\mathbf{i}_Q$ two huge amounts of mass going in opposite direction give rise to a lot of cancellations, as the orientation of curves is taken into account. As a simple example suppose to have two distinct points $x_0\not= x_1$ and consider the measure
\[
Q=\frac{1}{2}\,\delta_{\gamma_1}+\frac{1}{2}\,\delta_{\gamma_2},
\]
with $\gamma_1(t)=(1-t)\,x_0+t\,x_1$ and $\gamma_2(t)=(1-t)\,x_1+t\, x_0$. By computing the traffic intensity we obtain
\[
i_Q= \mathscr{H}^1\, \llcorner\, \overline{x_0 x_1},
\] 
which takes into account the intuitive fact that on the segment $\overline{x_0 x_1}$ globally there is a non negligible amount of transiting mass. On the other hand it is easily seen that 
\[
\mathbf{i}_Q\equiv 0.
\] 
\end{oss}
\noindent
Given a Radon measure $T$ on $\Omega$ we define the following space
$$
 \mathcal{Q}_p(T):=\Big\{Q\in\mathscr{M}_+(\mathcal L(\Omega)):\: i_Q\in L^p(\Omega)\ \text{ and }\ ({e_1-e_0})_\# Q=T\Big\},
$$
where $e_i:\mathcal{L}(\Omega)\to\Omega$ is defined by $e_i(\gamma)=\gamma(i)$, for $i=0,1$.
\noindent
Now consider a Carath\'eodory function $\mathcal H:\Omega\times\mathbb{R}^+\to\mathbb{R}^+$ such that
\begin{equation}
\label{pgrowth}
\lambda\, (t^p-1)\le\mathcal H(x,t)\le \frac{1}{\lambda}\, (t^p+1),\qquad x\in\Omega,\, t\in\mathbb{R}^+
\end{equation}
for some $0< \lambda\le 1$ and such that
\[
t\mapsto \mathcal H(x,t)\qquad \mbox{ is convex, for every }  x\in\Omega.
\]
If $\mathcal Q_p(T)\neq \emptyset$ then we define the following minimization problem:
\begin{equation}
\label{wardrop_}
\inf_{Q\in\mathcal Q_p(T)}\int_{\Omega}\mathcal H(x,i_Q(x))\, dx.
\end{equation}
\begin{oss}
Similar Lagrangian formulations have been studied in connection with transport problems involving {\it concave costs}, e.g. problems where to move a mass $m$ of a length $\ell$ costs $m^\alpha\, \ell$ ($0<\alpha<1$). For these the reader is referred to the monograph \cite{BCM}, as well as to the papers \cite{PS3,Va}.
\end{oss}
We show that problem \eqref{wardrop_} is well-posed and equivalent to the one in \eqref{beckmann}.
To this aim we use the following result, enunciated in the appendix (see Theorem \ref{smirnov}) in a slightly stronger formulation in terms of currents.
\begin{prop}
\label{smirnovlp}
Let $1\le p\le \infty$. Assume that $V\in L^p(\Omega, \mathbb R^N)$ and that it is acyclic. Let  $T=-\mathrm{div}\, V$ be a Radon measure on $\Omega$. It is then possible to find $Q\in\mathscr{M}_+(\mathcal{L}(\Omega))$ such that
\[
(e_0)_\# Q=T_-\qquad \mbox{ and }\qquad (e_1)_\# Q=T_+.
\] 
Moreover we have
\[ 
\mathbf i_Q=V \qquad\mbox{ and }\qquad i_Q=|V|.
\]
In particular $Q\in\mathcal{Q}_p(T)$.
\end{prop}
Thanks to the above result we can prove the following.
\begin{prop}
Let $T$ be a Radon measure on $\Omega$. The set \( \mathcal{Q}_p(T)\) is not empty if and only if $T\in \dot W^{-1,p}(\Omega)$. 
\end{prop}
\begin{proof}
Let us suppose that $T\not\in \dot W^{-1,p}(\Omega)$ and assume by contradiction that there exists $Q_0\in\mathcal{Q}_p(T)$. In particular
\begin{equation}
\label{zero}
\int_\Omega |i_{Q_0}|^p\, dx<+\infty.
\end{equation}
The vector measure $\mathbf{i}_{Q_0}$ satisfies \eqref{neumann}, since
\[
\begin{split}
\int_\Omega \nabla\varphi\cdot d\mathbf{i}_{Q_0} &=\int_{\mathcal{L}(\Omega)} \left(\int_0^1 \nabla \varphi(\gamma(t))\cdot \gamma'(t)\, dt\right)\, dQ_0(\gamma)\\
&=\int_{\mathcal{L}(\Omega)} \Big[\varphi(\gamma(1))-\varphi(\gamma(0))\Big]\, dQ_0(\gamma)=\langle T,\varphi\rangle,
\end{split}
\]
for every $\varphi\in C^1(\Omega)$.
Thanks to the fact that $|\mathbf{i}_{Q_0}|\le i_{Q_0}$ and to \eqref{zero} we have that $\mathbf{i}_{Q_0}\in L^p(\Omega;\mathbb{R}^N)$. This contradicts the fact that $T\not\in \dot W^{-1,p}(\Omega)$, as desired.
\vskip.2cm\noindent
Now take $T\in \dot W^{-1,p}(\Omega)$. Then there exists a minimizer $V$ of problem \eqref{beckmann} with $\mathcal{H}(x,z)=|z|^p$.
Thanks to Proposition \ref{acyclicbeck} we know that $V$ is acyclic. Since $T$ is a Radon measure we can apply Proposition \ref{smirnovlp} and we infer the existence of $Q\in\mathcal{Q}_p(T)$. This gives directly the thesis.
\end{proof}
We now prove our equivalence statement, which is the main result of this section. Observe that we prove at the same time existence of a minimizer for \eqref{wardrop_}.
\begin{teo}
\label{th:wardbeck}
Let $\mathcal H:\Omega\times\mathbb{R}^+\to\mathbb{R}^+$ be a Carath\'eodory function satisfying \eqref{pgrowth} and such that
\[
t\mapsto \mathcal H(x,t)\ \mbox{ is strictly convex and increasing},\qquad x\in\Omega.
\]  
If $T$ is a Radon measure belonging to $\dot W^{-1,p}(\Omega)$ then we have
\begin{equation}
\label{uguali}
\inf_{Q\in\mathcal Q_p(T)} \int_{\Omega}\mathcal H(x,i_Q)\,dx=\min_{V\in L^p(\Omega;\mathbb{R}^N)}\left\{\int_\Omega \mathcal H(x,|V|)\, dx\, :\, \begin{array}{c}-\mathrm{div\,}V=T,\\ V\cdot\nu_\Omega=0\end{array}\right\}
\end{equation}
and the infimum on the left-hand side is achieved. 
\par
Moreover, if $Q_0\in \mathcal Q_p(T)$ is optimal then $\mathbf{i}_{Q_0}\in L^p(\Omega;\mathbb{R}^N)$ is a minimizer of Beckmann's problem. Conversely, if $V_0$ is optimal then there exists $Q_{V_0}\in\mathcal{Q}_p(T)$ such that $i_{Q_{V_0}}=|\mathbf{i}_{Q_{V_0}}|$ minimizes the Lagrangian problem.
\end{teo}
\begin{proof}
By the previous result the set $\mathcal{Q}_p(T)$ is not empty. For every admissible $Q$ we have $|\mathbf{i}_Q|\le i_Q$, therefore $\mathbf{i}_Q$ is admissible for Beckmann's problem. Using the monotonicity of $\mathcal H(x,\cdot)$ we then obtain
\[
\min_{V\in L^p(\Omega;\mathbb{R}^N)}\left\{\int_\Omega \mathcal H(x,|V|)\, dx\, :\, \begin{array}{c}-\mathrm{div\,}V=T,\\ V\cdot\nu_\Omega=0\end{array}\right\}\le \inf_{Q\in\mathcal Q_p(T)} \int_{\Omega}\mathcal H(x,i_Q(x))\,dx<+\infty.
\]
Now let $V_0\in L^p(\Omega;\mathbb{R}^N)$ be a minimizer for Beckmann's problem. By Proposition \ref{acyclicbeck} $V_0$ is acyclic. Thus by Proposition \ref{smirnovlp} there exists $Q_0\in\mathcal{Q}_p(T)$ such that $|V_0|=i_{Q_0}$, i.e.
\[
\min_{V\in L^p(\Omega;\mathbb{R}^N)}\left\{\int_\Omega \mathcal H(x,|V|)\, dx\, :\, \begin{array}{c}-\mathrm{div\,}V=T,\\ V\cdot\nu_\Omega=0\end{array}\right\}=\int_{\Omega}\mathcal H(x,i_{Q_0}(x))\,dx.
\]
This shows that \eqref{uguali} holds true and that the infimum in the left-hand side is indeed a minimum.
\par
The relation between minimizers of the two problems is an easy consequence of the previous constructions.
\end{proof}
\appendix
\section{Decompositions of acyclic $1-$currents}
\label{currents}
\subsection{Definitions and links to vector fields}
The classical references which we use for currents are \cite{F,GMS}. We translate however all results in the language of the previous sections. In what follows, by $\Omega\subset\mathbb{R}^N$ we still denote the closure of an open bounded connected set having smooth boundary.
\begin{defi}
A \emph{$0-$current} on $\Omega$ is a distribution on $\Omega$ in the usual sense.
A \emph{$1-$current} on $\Omega$ is a vector valued distribution on $\Omega$. The relevant duality is the one with $1$-forms $\omega(x)=\sum_{i=1}^N\omega_i(x)\,dx^i$ having smooth coefficients, i.e. $\omega_i\in C^\infty(\Omega)$. We denote by $C^\infty (\Omega,\wedge^1\mathbb R^N)$ the space of such forms. More generally a $k-$current is an element in the dual of smooth $k-$forms $C^\infty(\Omega,\wedge^k\mathbb R^N)$.
\end{defi}
The above definition automatically gives the space of currents a natural weak topology, defined via the duality with smooth forms. 
For any current there is a natural definition of boundary. 
\begin{defi}
 If $I$ is a $k-$current on $\Omega$ then we can define its {\it boundary} $\partial I$ to be the $(k-1)$-current on $\Omega$ which satisfies
$$
\langle\partial I,\varphi\rangle=-\langle I,d\varphi\rangle,\qquad \mbox{ for all }\varphi\in C^\infty(\Omega,\wedge^{k-1}\mathbb R^N),
$$
where $d$ is the exterior derivative. For example if $k=1$ we must take $\varphi\in C^\infty(\Omega)$ and $d\varphi$ is the $1$-form $\sum_i \partial_{x_i}\varphi\, dx_i$.
\end{defi}
\begin{defi}
Let $I$ be a $k-$current. The \emph{mass} of $I$ is defined as
\[
\mathbb{M}(I)=\sup\left\{|\langle I,\omega\rangle|\, :\, \omega\in C^\infty(\Omega,\wedge^k\mathbb R^N),\:\sup_{x\in\Omega} \|\omega(x)\|\le 1\right\},
\]
where the norm $\|\omega\|$ for an alternating $k$-tensor $\omega$ is defined as 
$$
\|\omega\|=\sup\{\langle\omega,\mathbf{e}\rangle:\:\mathbf{e}\text{ unit simple }k-\text{vector}\}.
$$
For $k=1$ this coincides with the usual norm $\|\omega\|=\sqrt{\omega_1^2 +\ldots +\omega_N^2}$.
\end{defi}
We use just $1-$currents which are distributions of order $0$ (i.e. vector valued Radon measures) and in their boundaries (which are scalar distributions of order $1$). On this point a comment is in order.
\begin{oss}
\label{distrrepint}
Finite mass $1-$currents can be identified with vector-valued Radon measures as follows. To every smooth $1-$form $\omega$ we may associate naturally a vector field $X_\omega:=(\omega_1,\ldots,\omega_N)$. We can then write for a $1-$current $I$
$$
\int X_\omega\, dI:=\langle I,\omega\rangle.
$$
Since $C^\infty$ is dense in $C^0$ the resulting linear functional on smooth vector fields can be identified via Hahn-Banach theorem to a unique linear functional on $C^0$ vector fields. The latter is indeed a vector-valued Radon measure by Riesz representation theorem. Since $\|\omega(x)\|=\|X_\omega(x)\|$ by the above definition, we automatically obtain that 
$$
\mathbb M(I)=\int_\Omega d|I|,
$$
i.e. the mass equals the total variation of $I$ regarded as a Radon measure. The same reasoning can be applied to $0-$currents of finite mass, by identifying them with scalar Radon measures\footnote{See also ``{\it Distributions representable by integration}'' in \cite[4.1.7]{F}}.
\end{oss}
\begin{defi}[variation of a current]
Let $A$ be a $k-$current with $k\in\{0,1\}$. Then we may define the variation measure $\mu_A$ of $A$ in the usual sense by identifying $A$ with a Radon measure as in Remark \ref{distrrepint}. Thus for a Borel set $E$ we define
$$
\mu_A(E):=\sup\left\{\sum_{i=1}^k\left|\int_{E_i}dA\right|:E_i\text{ form a Borel partition of } E\right\}.
$$
\end{defi}
An equivalent way of defining $\mu_A$ would be as the infimum of all measures $\mu$ such that $\langle A,\omega\rangle\leq\int_\Omega \|\omega\|d\mu$ for all smooth $1$-forms $\omega$.
\vskip.3cm\noindent
We recall that a $k-$current $T$ is said to be {\it normal} if\,\footnote{For $k=0$ we define $\partial A=0$ and thus the condition on $\partial A$ can be omitted.}
\[
\mathbb{M}(T)+\mathbb{M}(\partial T)<+\infty. 
\]
We now define {\it flat currents}, a class useful for its closure properties.
\begin{defi}
 We define the \emph{flat norm} of a $k-$current $A$ as follows
\[
\mathbb{F}(A)=\inf\left\{ \mathbb{M}(A-\partial I)+\mathbb{M}(I)\, :\, I \mbox{ is a $(k+1)-$current with } \mathbb{M}(I)<\infty\right\}.
\]
Then the space of \emph{flat $k-$currents} is defined as the completion of normal $k-$currents in the flat norm.
\end{defi}
Flat currents of finite mass have the following characterization, which will be exploited in the sequel.
\begin{lm}
\label{lm:bemolle}
Let $T$ be a $k-$current of finite mass. Then $T$ is flat if and only if there exists a sequence of normal $k-$currents $\{T_n\}_{n\in\mathbb{N}}$ such that
\[
\lim_{n\to\infty} \mathbb{M}(T_n-T)=0.
\]
\end{lm}
\begin{proof}
This is a standard fact but we provide a proof for the sake of completeness. By definition of flat convergence there exists a sequence $\{I_n\}_{n\in\mathbb{N}}$ of normal $k-$currents and a sequence $\{Y_n\}_{n\in\mathbb{N}}$ of $(k+1)-$currents such that
\[
\lim_{n\to\infty} \left[\mathbb{M}(T-I_n-\partial Y_n)+\mathbb{M}(Y_n)-\frac{1}{n}\right]\le \lim_{n\to\infty}\mathbb{F}(T-I_n)=0.
\]
Then we set $T_n=I_n+\partial Y_n$, which by construction is a $k-$current and 
\[
\lim_{n\to\infty} \mathbb{M}(T-T_n)=0,
\] 
thanks to the previous estimate.
To prove that $T_n$ is a normal current we first write
\[
\mathbb{M}(T_n)\le\mathbb{M}(I_n)+\mathbb{M}(\partial Y_n)\le 2\,\mathbb{M}(I_n)+\mathbb{M}(T-I_n-\partial Y_n)+\mathbb{M}(T)<+\infty.
\]
Observe that we used that $T$ has finite mass and the triangular inequality. Secondly we note that
\[
\mathbb{M}(\partial T_n)=\mathbb{M}(\partial I_n)<+\infty,
\]
since $\partial(\partial Y_n)=0$. 
\vskip.2cm\noindent
The converse implication is simpler: by definition of flat norm we have
\[
\mathbb{F}(T-T_n)\le\mathbb{M}(T-T_n).
\]
This concludes the proof.
\end{proof}
A significant instance of flat $1-$currents with finite mass is given by $L^1$ vector fields.
\begin{lm}
\label{lm:flatL1}
Given $V\in L^1(\Omega;\mathbb{R}^N)$ we naturally associate to it the $1-$current $I_V$ of finite mass defined by
\[
\langle I_V,\omega\rangle=\sum_{i=1}^N \int_\Omega \omega_i\, V_i\, dx:=\int_\Omega\omega(V) \qquad \mbox{ for every }\ \omega\in C^\infty(\Omega,\wedge^1 \mathbb{R}^N).
\]
This current has compact support contained in $\Omega$ and $\mathbb{M}(I_V)=\|V\|_{L^1}$.
Moreover $I_V$ is a flat current.
\end{lm}
\begin{proof}
We just prove that $I_V$ is a flat current, the first statement being straightforward. To this aim we use the characterization of Lemma \ref{lm:bemolle} and we construct
the approximating currents by convolution. For every $\varepsilon\ll 1$ we define
\[
\Omega_\varepsilon=\{x\in\Omega\, :\, \mathrm{dist}(x,\partial\Omega)>2\, \varepsilon\}.
\]
Then we take a standard convolution kernel $\varrho\in C^\infty_0$ supported on the ball $\{x\,:\, |x|\le 1\}$ and we define
\[
\varrho_\varepsilon(x)=\varepsilon^{-N}\, \varrho\left(\frac{x}{\varepsilon}\right),\qquad x\in\mathbb{R}^N. 
\]
We also set 
$$
V_\varepsilon:=(V\cdot 1_{\Omega_\varepsilon})\ast\varrho_\varepsilon,
$$
where $1_E$ stands for the characteristic function of a set $E$. Define now $I_\varepsilon:=I_{V_\varepsilon}$ and observe that $V_\varepsilon$ (and thus $I_\varepsilon$) has compact support contained in $\Omega$. From the mass estimate and by H\"older inequality we obtain that masses are equi-bounded, since
$$
\mathbb M(I_\varepsilon)\leq \|V_\varepsilon\|_{L^1}\leq C\|V\|_{L^1}.
$$
The boundedness of $\partial I_\varepsilon$ follows via a similar strategy. As $V_\varepsilon$ has compact support (strictly contained) in $\Omega$, we have
\[
\begin{split}
|\langle \partial I_\varepsilon,\varphi\rangle| =|\langle I_\varepsilon,d\,\varphi\rangle|&=\left|\int_\Omega V_\varepsilon\cdot\nabla \varphi\, dx\right|\\
&=\left|\int_\Omega \mathrm{div}\,V_\varepsilon\, \varphi\, dx\right|\le \|\mathrm{div}\, V_\varepsilon\|_{L^1}\, \|\varphi\|_{L^\infty}.
\end{split}
\]
Setting $C_\varepsilon=\|\mathrm{div}\, V_\varepsilon\|_{L^1}$ and passing to the supremum on $\varphi$ we obtain
\[
\mathbb{M}(\partial I_\varepsilon)\le C_\varepsilon<+\infty.
\]
This implies that $\{I_\varepsilon\}_{\varepsilon>0}$ is a sequence of normal currents. The mass convergence $\mathbb{M}(I_\varepsilon-I)$ easily follows from the convergence of $V_\varepsilon$ to $V$ in $L^1(\Omega;\mathbb{R}^N)$.
\end{proof}
\begin{oss}
It is easily seen that the boundary $\partial I_V$ corresponds to the distributional divergence of $V$ i.e.
\[
\langle \partial I_V,\varphi\rangle=-\int_\Omega \nabla \varphi\cdot V\, dx
\qquad \mbox{ for every } \varphi\in  C^\infty(\Omega).
\]
This distribution has compact support in $\Omega$ as well. 
\end{oss}
\begin{defi}
A $1-$current $I$ is called {\it acyclic} if whenever we can write $I=I_1+I_2$, with $\mathbb{M}(I)=\mathbb{M}(I_1)+\mathbb{M}(I_2)$ and $\partial I_1=0$, there must result $I_1=0$.
\end{defi}
For $I=I_V$ with $V\in L^1$ we have the correspondence with Definition \ref{Vacyclic}, i.e. 
\begin{equation}
\label{aciclico}
V\text{ is acyclic}\Longleftrightarrow I_V\text{ is acyclic}.
\end{equation}
\subsection{Lipschitz curves as currents}
We recall that $\mathcal L(\Omega)$ is the space of equivalence classes of Lipschitz curves $\gamma:[0,1]\to\Omega$ (see the beginning of Section \ref{sec:lagrangian}), with the topology of uniform convergence.\\
Here we remark that to each $\gamma\in\mathcal L(\Omega)$ we may associate a vector valued distribution i.e. a $1-$current, denoted by $[\gamma]$ and defined by requiring 
$$
\langle [\gamma],\omega\rangle:=\int_{\gamma}\omega=\int_0^1\omega(\gamma(t))\, [\gamma'(t)]\,dt=\sum_{i=1}^N\int_0^1\omega_i(\gamma(t))\,\gamma'_i(t)\,dt,
$$
for all $\omega\in C^\infty(\Omega,\wedge^1 \mathbb{R}^N)$.
Note that this expression is well-defined on $\mathcal L(\Omega)$ since the integral on the right is invariant under reparameterization.
\vskip.2cm\noindent
For $\gamma\in\mathcal L(\Omega)$ we have $\mathbb M([\gamma])\leq\ell(\gamma):=\int_0^1 |\gamma'(t)|\, dt$ with equality exactly when $\gamma$ has a representative which is injective for $\mathscr H^1$-almost every time. We are interested in a stronger requirement, namely that the curve does not even intersect itself. In this case an injective representative exists (such curves are called ``arcs'' in \cite{PS}). We fix a notation for such classes of curves.
\begin{defi}[Arcs]
 We define $\widetilde{\mathcal L}(\Omega)$ the subset of $\mathcal L(\Omega)$ made of those classes of curves $\gamma:[0,1]\to\Omega$ which have an injective representative.
\end{defi}
\begin{oss}
 One could think of the above-defined arcs as ``acyclic curves'', where a ``cycle'' can mean two things: 
\begin{itemize}
\item we can have a cycle in the parameterization, where a cycle would be represented by a curve satisfying (up to reparameterization) $\gamma(t)=\gamma(1-t)$ and ``inserting a cycle of type $\gamma$'' in another curve $\overline\gamma$ such that $\overline\gamma(t_0)=\gamma(0)$ would result into the curve:
$$
\widetilde\gamma(s)=
\left\{
\begin{array}{ll}
\overline\gamma(2s)&\text{ if }s\in[0,t_0/2]\\
\gamma(2s-t_0)&\text{ if }s\in[t_0/2,(t_0+1)/2]\\
\overline\gamma(2s-1)&\text{ if }s\in[(t_0+1)/2,1].
\end{array}
\right.
$$
\item a curve which intersects itself (i.e. which has no injective parameterization) gives instead rise to a $1-$current which is not acyclic, since it has a reparameterization containing an injectively parameterized loop.
\end{itemize}
The fact that in the decomposition of acyclic currents one restricts to using just arcs (for which neither type of cycle occurs) is then another natural consequence of the robustness of the acyclicity requirement.
\end{oss}
On $1-$currents we consider the topology of distributions. The following result links the two topologies:
\begin{lm}
 \label{curvescurrents}
The map $\mathcal L(\Omega)\ni\gamma\mapsto[\gamma]$ as defined above is continuous on the sublevels of the length functional $\ell:\mathcal L(\Omega)\to\mathbb R$.
\end{lm}
\begin{proof}
Assuming that $\gamma_i\to\gamma$ and $\ell(\gamma_i)\leq C$ we then obtain that $\gamma_i$ converge uniformly. In particular they converge as distributions.
\end{proof}
\subsection{Smirnov decomposition theorem}

We can now state the theorem on the decomposition of $1-$currents due to Smirnov \cite{S} and recently extended by Paolini and Stepanov in \cite{PS,PS2} to metric spaces. \begin{teo}
\label{smirnov}
 Suppose that $I$ is a normal acyclic $1-$current on $\Omega$. Then there exists $Q\in\mathscr{M}_+(\mathcal{L}(\Omega))$ concentrated on $\widetilde{\mathcal L}(\Omega)$ and such that the following decompositions of $I$ are valid in the sense of distributions:
\begin{equation}
\label{corrente}
I=\int_{\mathcal L(\Omega)}[\gamma]\,d Q(\gamma)\qquad \mbox{ and }\qquad\mu_I=\int_{\mathcal L(\Omega)}\mu_{[\gamma]}\,d Q(\gamma) ,
\end{equation}
\begin{equation}
\label{bordo}
\partial I=\int_{\mathcal L(\Omega)}\partial[\gamma]\,d Q(\gamma)\quad \mbox{ and }\qquad  \mu_{\partial I}=\int_{\mathcal L(\Omega)}\mu_{\partial[\gamma]}\, d Q(\gamma).
\end{equation}
\end{teo}
We now note down some reformulations of the items present in the above theorem in terms of measures and vector fields:
\begin{itemize}
 \item the total variation is the mass norm, i.e. $\mu_I(\Omega)=\mathbb M(I)$ and $\mu_{\partial I}(\Omega)=\mathbb M(\partial I)$;
\vskip.2cm
 \item if $V$ is a $L^1_{loc}$ vector field then $I_V$ has variation measure $\mu_{I_V}=|V|\cdot \mathscr{L}^N$;
\vskip.2cm
 \item if $\rho=\rho_+-\rho_-$ is the decomposition of a signed Radon measure into positive and negative part then $\mu_\rho=\rho_++\rho_-$;
\vskip.2cm
\item in particular for $\gamma\in\mathcal L(\Omega)$ we have $\mu_{\partial[\gamma]}=\delta_{\gamma(1)} + \delta_{\gamma(0)}$. Since this measure has total variation $2$ for all $\gamma$ we can quantify the total mass of the above $Q$ by means of the mass norm of the boundary of $I$. Namely, we have
$$
Q(\mathcal L(\Omega))=\frac{1}{2}\,\mu_{\partial I}(\Omega)=\frac{\mathbb M(\partial I)}{2};
$$
 \item for $\gamma\in \widetilde{\mathcal{L}}(\Omega)$ there holds $\mu_{[\gamma]}=\mathscr H^1\,\llcorner\,\mathrm{Im}(\gamma)$, i.e. this is the arclength measure of $\gamma$;
\vskip.2cm
\item by expanding the definitions and comparing to Section \ref{sec:lagrangian} we see that for $I=I_V$ with $V\in L^1(\Omega)$ and for $Q$ as in Theorem \ref{smirnov}, there holds 
\[
|V|\cdot\mathscr L^N=\mu_{I_V}=i_Q\cdot\mathscr L^N,\qquad V\cdot\mathscr L^N=I=\mathbf{i}_Q\cdot\mathscr L^N
\]
and
\[
\op{div}\,V=\partial I_V=(e_0 - e_1)_\# Q.
\]
\end{itemize}
\vskip.2cm
All these reformulations allow to translate Theorem \ref{smirnov} in the case of $I=I_V$, with $V\in L^p(\Omega)$ for $p\geq1$. This is the content of the next result. 

\begin{coro}[Reformulation of Theorem \ref{smirnov}]
\label{coro:smirnov}
Suppose that $V\in L^1(\Omega)$ is an acyclic vector field such that $\op{div}\,V$ is a Radon measure. Then there exists $Q\in\mathscr{M}_+(\mathcal{L}(\Omega))$ concentrated on $\widetilde{\mathcal L}(\Omega)$ and such that the following decompositions of $V$ are valid in the sense of distributions:
\begin{equation}
\label{V}
\mathbf{i}_Q=V\qquad \mbox{ and }\qquad i_Q=|V|,
\end{equation}
\begin{equation}
\label{Vbordo}
-\op{div}\,V=(e_1 - e_0)_\# Q\qquad \mbox{ and }\qquad (-\op{div}\,V)_+ +(-\op{div}\,V)_-=
(e_1 + e_0)_\# Q.
\end{equation}
\end{coro}

\subsection{The case of flat currents}

In Section \ref{sec:lagrangian} we required $T\in \dot W^{-1,p}(\Omega)$ to be a Radon measure. As already mentioned, this further hypothesis permits to identify optimal vector fields for Beckmann's problems with acyclic normal currents. Then well-posedness and equivalence of the problems can be obtained by means of Smirnov's Theorem. However in the setting of Beckmann's problem and of its dual it would be natural to allow $T$ to be a generic element of $\dot W^{-1,p}(\Omega)$. If one whishes to extend the analysis of the Lagrangian formulation to this larger space then one is naturally lead to consider a possible extension of Smirnov's result to $L^1$ vector fields having divergence which is not a Radon measure. Observe that {\it such vector fields correspond to flat currents} (see Lemma \ref{lm:flatL1}). In this subsection we investigate the possibility to have Smirnov's Theorem for such a class of currents.
\vskip.2cm\noindent
We start by observing that the measure $Q$ which decomposes $I$ may not be finite in general.  
\begin{exa}
 \label{curvette}
For $1\le p<\frac{N}{N-1}$ we consider an infinite sequence of small dipoles $\left\{(a_i,b_i)\right\}_i$ such that
$$
\sum_{i=1}^\infty|a_i-b_i|^{N-p(N-1)}<+\infty, \qquad\text{ and } D_{a_i,b_i}\text{ are disjoint},
$$
where the sets $D_{a_i,b_i}$ are defined as in Example \ref{curvette1}. 
If we consider the vector fields $V_{a_i,b_i}$ as in Example \ref{curvette1} then the new vector field defined by $V=\sum_{i=1}^\infty V_{a_i,b_i}$ verifies
$$
\|V\|_{L^p}^p=\left\|\sum_{i=1}^\infty V_{a_i,b_i}\right\|^p_{L^p}\leq C_N\, \sum_{i=1}^\infty |a_i-b_i|^{N-p\,(N-1)}<+\infty,
$$
which implies $T:=\sum_i(\delta_{a_i} - \delta_{b_i})\in \dot W^{-1,p}(\Omega)$.
By observing that $\infty=\mathbb M(T)=\int_{\mathcal L(\Omega)}dQ$ for any decomposing measure we see that no finite measure $Q$ can be found. On the other hand a $\sigma-$finite measure $Q$ can be found, since each $V_{a_i,b_i}$ can be separately decomposed with a measure $Q_i$ of mass $2$ and the $Q_i$ have disjoint supports. 
\end{exa}
\begin{exa}
We present now another version of Example \ref{curvette}, which exploits the Sobolev embedding theorem. Let us take again $1\le p<N/(N-1)$, that is $q=p/(p-1)>N$. Then $\dot W^{1,q}(\Omega)$ can be identified with a space of functions which are H\"older continuous of exponent $\alpha=1-N/q$. We consider the following two curves
\[
\gamma_1(t)=\frac{1}{t^{2/\alpha}}\, (\cos t,\sin t)\qquad \mbox{ and }\qquad\gamma_2(t)=\frac{g(t)}{t^{2/\alpha}}\, (\cos t,\sin t),\qquad t\ge 1,
\]
where $g:[1,\infty)\to\mathbb{R}^+$ is a continuous function such that 
\[
1>g(t)>\left(\frac{t}{t+2\, \pi}\right)^{2/\alpha}\qquad \mbox{ and }\qquad t\mapsto \frac{g(t)}{t^{2/\alpha}} \mbox { is decreasing}.
\]
We define the distribution
\[
\langle T,\varphi\rangle=\int_1^\infty \Big[\varphi(\gamma_1(t))-\varphi(\gamma_2(t))\Big]\, dt,\qquad \varphi\in C^\infty(\Omega).
\]
This is an element of $\dot W^{-1,p}(\Omega)$ since by Sobolev embedding $[\varphi]_{C^{0,\alpha}}\le C_\Omega\,\|\varphi\|_{\dot W^{1,q}}$.  Thus
\[
\begin{split}
|\langle T,\varphi\rangle|\le \int_1^\infty |\varphi(\gamma_1(t))-\varphi(\gamma_2(t))|\, dt&\le C\, \|\varphi\|_{W^{1,q}(\Omega)}\, \int_1^\infty |\gamma_1(t)-\gamma_2(t)|^\alpha\, dt\\
&=C\,\|\varphi\|_{W^{1,q}(\Omega)}\,\int_1^\infty \frac{|1-g(t)|^\alpha}{t^2}\, dt\\
&\le 2^\alpha\, C\,\|\varphi\|_{W^{1,q}(\Omega)},\qquad \varphi\in \dot W^{1,q}(\Omega).
\end{split}
\]
We then introduce the measure on paths $Q_T$ defined by
\[
Q_T=\int_1^\infty\, \delta_{\,\overline{\gamma_1(t)\, \gamma_2(t)}}\, dt,
\]
where for $t\ge1$ we indicate by $\overline{\gamma_1(t)\, \gamma_2(t)}$ the straight segment going from $\gamma_1(t)$ to $\gamma_2(t)$.
Observe that for every $\varphi$ we have
\[
\int_{\mathcal{L}(\overline\Omega)} [\varphi(\gamma(0))-\varphi(\gamma(1))]\, dQ_T(\gamma)=\int_1^\infty \Big[\varphi(\gamma_1(t))-\varphi(\gamma_2(t))\Big]\, dt=\langle T,\varphi\rangle
\]
and
\[
\begin{split}
\int_{\mathcal{L}(\overline\Omega)} \ell(\gamma)\, dQ_T(\gamma)=\int_1^\infty |\gamma_1(t)-\gamma_2(t)|\, dt&=\int_1^\infty \frac{|1-g(t)|}{t^{2/\alpha}}\, dt\\
&\le \frac{2\,\alpha}{\alpha-2}\, t^{1-\frac{2}{\alpha}}\Big|_{1}^\infty=\frac{2\,\alpha}{2-\alpha}<\infty,
\end{split}
\]
while $Q_T$ is not finite, but just $\sigma$-finite.
\end{exa}
The previous examples clarify that we cannot hope to give a distributional meaning to the positive and negative parts of the divergence of $V$. The good definition of $\op{div}\,V$ as a distribution relies in general on some sort of ``almost--cancellation''. Therefore the last no-cancellation requirement $(-\op{div}\,V)_+ +(-\op{div}\,V)_-=(e_1 + e_0)_\# Q$ 
of Smirnov's Theorem \ref{smirnov} should be relaxed when we try to extend it to a larger class of $V$'s.
\vskip.2cm\noindent
Actually we can say more. For general flat currents even the existence of a (possibly $\sigma$-finite) decomposition $Q$ satisfying just \eqref{corrente} is not granted, as shown in the next example.

\begin{exa}
Let $\Omega=[0,1]^N\subset\mathbb{R}^N$ and let us consider a totally disconnected closed set $E\subset \Omega$ such that $\mathscr{L}^N(E)>0$. We then pick a vector $\vartheta_0\in \mathbb{R}^N\setminus\{0\}$ and set
\[
V(x)=\vartheta_0\cdot 1_E(x),\qquad x\in\Omega.
\]
Of course this is an $L^\infty(\Omega)$ vector field and thanks to Lemma \ref{lm:flatL1} we have that the associated $1-$current $I_V$ is flat. We claim that {\it $I_V$ does not admit a Smirnov decomposition}. 
Indeed, assume by contradiction that a decomposition $Q\in\mathscr{M}_+(\mathcal{L}(\Omega))$ satisfying the following condition (equivalent to \eqref{corrente}) exists:
\begin{equation}
 \label{correntecantor}
\mu_{I_V}=\int_{\mathcal L(\Omega)} \mu_{[\gamma]}\, dQ(\gamma).
\end{equation}
We see that $\op{spt}(\mu_{I_V})=\op{spt}(I_V)$ and from this we can infer that $Q-$a.e. curve $\gamma$ has support included in $\op{spt}(I_V)=E$ (see also \cite[Remark 5]{S}).
Indeed, if the latter were not true then the supports of the two sides of \eqref{correntecantor} would differ, as follows by observing that the left-hand side is a superposition of the positive measures $\mu_{[\gamma]}$. 
\par
By knowing now that $Q-$a.e. curve $\gamma$ has support in the totally disconnected set $E$ and that the curves $\gamma$ are connected, we deduce that $Q-$a.e. curve $\gamma$ is constant. This implies that 
\[
\mbox{for $Q-$a.e. curve $\gamma$ there holds }[\gamma]=0.
\]
Therefore from \eqref{correntecantor} it follows $\mu_{I_V}=0$, which contradicts the fact that 
\[
 \mu_{I_V}(\Omega)=\mathbb M(I_V)=\int_\Omega |V|\, dx>0.
\]
Since the existence of $Q$ satisfying \eqref{correntecantor} leads to a contradiction, we conclude that no Smirnov decomposition of $I_V$ exists. Note that whether $Q$ is assumed to be finite or only $\sigma-$finite is immaterial for this contradiction.
\par
We point out that by defining the distribution  $T$ as
\[
\langle T,\varphi\rangle=\int_\Omega V(x)\cdot \nabla \varphi(x)\, dx=\int_E \vartheta_0\cdot \nabla \varphi(x)\, dx,\qquad \mbox{ for every }\varphi\in C^1(\Omega),
\] 
this can be considered as an element of $\dot W^{-1,p}(\Omega)$ for any $1<p<\infty$, thanks to Lemma \ref{lm:cara}.
\par
We also claim that $V$ (and thus $I_V$) {\it is acyclic}. Suppose that we can write $V=V_1 +V_2$ with $|V|=|V_1|+|V_2|$ and $\op{div}\,V_1=0$. This implies that $V_1=\lambda_1\, \vartheta_0\cdot1_E$, where $\lambda_1\in L^1(\Omega)$ and it takes values in $[0,1]$. In particular as above we have
\[
0= \langle -\op{div}\,V_1,\varphi\rangle=\int_E \lambda_1(x)\, \vartheta_0\cdot \nabla \varphi(x)\, dx,\qquad \mbox{ for every }\varphi\,\in C^1(\Omega).
\]
By taking $\varphi(x)=\vartheta_0\cdot x$ we observe that the integral is nonzero unless $\lambda_1\equiv 0$ a.e. on $E$, in which case $V_1=0$. By appealing to Definition \ref{Vacyclic} we eventually prove that $V$ is acyclic.
\end{exa}
\begin{ack}
This work started during the conference ``Monge-Kantorovich optimal transportation problem, transport metrics and their applications'' held in St. Petersburg in June 2012. The authors wish to thank the organizers for the kind invitation. Guillaume Carlier and Eugene Stepanov are gratefully acknowledged for some comments on a preliminary version of the paper. L.B. has been partially supported by the {\it Agence National de la Recherche} through the project ANR-12-BS01-0014-01 {\sc Geometrya}.
\end{ack}

\end{document}